\documentclass[letterpaper,10pt,journal]{ieeetran}
\usepackage{graphicx}
\usepackage{gensymb}
\usepackage{amsmath, amssymb, amsthm}
\usepackage{pdfpages}
\usepackage{algorithm}
\usepackage[noend]{algpseudocode}
\usepackage{verbatim}
\usepackage{booktabs}
\usepackage{textcomp}
\newtheoremstyle{assumptionstyle}{0in}{0in}{\normalfont}{0.5em}{\itshape}{:}{.5em}{}
\newtheoremstyle{propertystyle}{0in}{0in}{\normalfont}{}{\bf}{:}{.5em}{}
\theoremstyle{plain}
\newtheorem{theorem}{Theorem}
\theoremstyle{assumptionstyle}
\newtheorem{assumption}{Assumption}
\theoremstyle{propertystyle}
\newtheorem{property}{Property}

\begin{document}
\setlength{\abovedisplayskip}{3pt}
\setlength{\belowdisplayskip}{3pt}

\title{Reactive Trajectory Generation in an Unknown Environment}

\author{Kenan Cole,
\thanks{Kenan Cole is a graduate student in the Mechanical and Aerospace Engineering Department,     The George Washington University, Washington, DC 20052, USA}
        ~Adam M. Wickenheiser%
\thanks{Adam M. Wickenheiser is an assistant professor in the Mechanical and Aerospace Engineering Department, The George Washington University,
        Washington, DC 20052, USA}}
        
\maketitle

\begin{abstract}
Autonomous trajectory generation for unmanned aerial vehicles (UAVs) in unknown environments continues to be an important research area as UAVs become more prolific. In this paper, we develop a trajectory generation algorithm for a vehicle in an unknown environment with wind disturbances, that relies only on the vehicle's on-board distance sensors and communication with other vehicles within a finite region to generate a collision-free trajectory that is continuous up to the fourth derivative. The proposed trajectory generation algorithm can be used in conjunction with high-level planners and low-level motion controllers, as demonstrated. The algorithm provides guarantees that the trajectory does not violate the vehicle's thrust limitation, sensor constraints, or a user-defined clearance radius around other vehicles and obstacles. Simulation results of a quadrotor moving through an unknown environment with moving obstacles demonstrates the trajectory generation performance. 
\end{abstract}

\section{Introduction}
The push for autonomous and beyond-line-of-sight (BLOS) operation of UAVs is becoming more of a reality with improved sensors both commercially \cite{Echodyne} and academically \cite{Newmeyer2016}. Our research examines formations of vehicles operating in unknown environments where the vehicles may be required to move relative to or independent of one another. Collision-free trajectory generation to a goal position for each vehicle is the focus of this paper.  

There are several approaches for trajectory generation in the presence of obstacles and/or vehicles, including global planners, local and reactive planners, and formation controllers. Global optimization techniques are prevalent \cite{Turpin2014,Bhattacharya2015,VanLoock2014} because they can ensure convergence on the goal position, assuming a known environment. This is not possible for applications where the environment is dynamic and unknown.

Local planners examine a shorter time window to reduce the computational expense and can address obstacles that may not be known a priori \cite{AlonsoMora2015,Shiller2013}.  One of the main drawbacks to the local planners is the lack of an overall safety or convergence guarantee since the optimization is occurring for short time windows for only the closest obstacles. 

Reactive controllers, which are a type of local planner, employ algorithms that generate the trajectory directly as the environment is sensed \cite{Chunyu2010, Tang2013, Matveev2015}. One drawback is that they do not guarantee smoothness of the trajectory. This is problematic because vehicle thrust constraints may be violated and higher derivatives may not be bounded, which can violate vehicle controller requirements. 

Formation controllers can provide solutions for collision avoidance with other vehicles in a variety of ways including global optimization where the environment must be known \cite{Weihua2011,Ayanian2010} or potential fields to guide the vehicles \cite{Chang2015}. In some cases avoidance is achieved by navigating the entire formation around the obstacle(s) \cite{Zhang2010, Sarkar2013, AlonsoMora2015}. For the present scenario, the formation can be of varying size and distribution, which is more similar to swarming behavior such as \cite{Belta2004}, which does not discuss obstacle avoidance, or \cite{Hedjar2014}, which relies on a distributed optimization to avoid obstacles and maintain the formation. In our scenario we seek to use the same trajectory generation for vehicles that have been re-tasked and are no longer part of the formation, so the avoidance must be applicable to obstacles and vehicles alike. 

In addition to collision avoidance, the vehicle's physical limitations such as sensor range (\cite{Tang2013}, \cite{Chunyu2010}, \cite{Choi2013}), maximum velocity (\cite{Matveev2015}, \cite{Tang2013}), clearance radius (\cite{AlonsoMora2015}, \cite{Matveev2015}, \cite{Choi2013}, \cite{Chunyu2010}), and turning rate (\cite{Choi2013}, \cite{Chunyu2010}) must be considered. All of these constraints impact the generation of a feasible trajectory, and to date, no trajectory planner accounts for all of these constraints simultaneously. 


Similarly, none of the cases examined consider the disturbance as input to the trajectory generation. Disturbance inclusion is much more prevalent in vehicle controllers to show ultimate bounded or asymptotic stability \cite{Cabecinhas2013,Waslander2009, Sydney2013,FischerNL2014}. In order to achieve these stability guarantees though, the controllers require that the desired trajectory higher derivatives exist and are bounded. To meet these criteria, the control authority to overcome the disturbance must also be considered.

Our goal is to address each of these areas: collision avoidance in unknown environments, smooth trajectories (and derivatives) that do not violate vehicle thrust or sensor constraints, inclusion of the bounded disturbance, and setting maximum velocity bounds. The problem definition, properties, and assumptions are given in Sec.~\ref{SecProbDef}. The trajectory generation is defined in Sec.~\ref{SecTrajGen}, describing the identification of potential collisions and the algorithm to adjust heading/velocity to clear the obstacle. Section \ref{SecSafetyGuarantee} provides the analysis for solving the trajectory curve timespan and bounding the vehicle's maximum safe cruise velocity. Section \ref{SecVehController} defines the vehicle dynamics and controller for the simulation case study presented in Sec.~\ref{SecSim}. Finally Sec.~\ref{SecConclusion} provides concluding remarks. 

\section{Problem Definition}
\label{SecProbDef}
We define a trajectory generation algorithm with the following properties and assumptions for an environment similar to Fig.~\ref{FigDetObsInPath}. 

\subsection{Algorithm Properties}
\label{SubSecAlgProps}
\begin{property} Generation of a smooth desired trajectory $\mathbf{p}_d \in \mathbb{R}^3$ where $\mathbf{p}_d^{(i)} \in \mathbb{R}^3,~\forall i=0,1,\ldots,n$ exist, are bounded, and respect the vehicle's maximum thrust, $f_{max}$, for a translational disturbance of unknown direction and bounded magnitude, $||\mathbf{d}_p|| \leq d_{p,max}$.
\end{property}
\begin{property} Clearance of all obstacles and vehicles by a user-defined clearance radius, $r_c$, which takes into account vehicle size, and measurement, estimation, and tracking errors.
\end{property}

\subsection{Algorithm Assumptions}
\begin{assumption}
\label{AssumptionPlanar}
	Vehicle desired trajectories and obstacle motions are planar, but vehicle dynamics are not restricted to be planar.
	\end{assumption}
	\begin{assumption}
	\label{AssumpVehicles}
	 Vehicles are finite in number and heterogeneous in physical parameters (mass, max thrust, etc) and importance (i.e. higher valued asset).  
	\end{assumption}
	\begin{assumption}
	 \label{AssumpSensorComms} Vehicles sensor and communication sample periods, $\Delta T_s = \Delta T_a$, and ranges, $r_s = r_a$ are finite, equal, and provide perfect information. 
	\end{assumption}
	\begin{assumption}
	 \label{AssumpVehCommsInfo} Vehicles share current position and heading information when in range using wireless communications.
	\end{assumption} 
	\begin{assumption}
	 Wind disturbances are bounded, time-varying, and planar.
	\end{assumption}
	\begin{assumption}
	 The clearance radius $r_c$ ensures there are no aerodynamic interactions between one vehicle and another or with obstacles.
	\end{assumption}
	
	\begin{assumption}
	\label{AssumpLessCapable}
	The obstacles are of finite size and number, in plane, and move with constant velocity (less than vehicle velocity) and heading. The obstacle separation does not prevent the vehicles from moving between them.
	\end{assumption}
	\begin{assumption}
	\label{AssumpGoalPos}
	 Goal positions are not too close to obstacles or each other to violate vehicle clearance radii and are not infinitely far from the coordinate origin.
	\end{assumption}
\begin{figure}
	\begin{center}
		\includegraphics[width=2in]{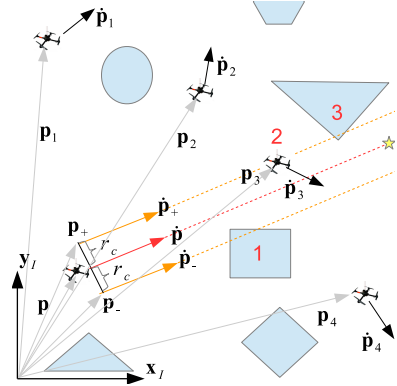}
		\caption{\label{FigDetObsInPath} Representative environment where a vehicle must navigate around other vehicles and obstacles to reach a goal position (yellow star). Two offset vectors, $\mathbf{p}_+$ and $\mathbf{p}_-$ are shown to account for the desired clearance radius, $r_c$. The vehicle prioritizes the potential collisions based on distance and heading angle.}
	\end{center}
\end{figure}

\section{Trajectory Generation}
\label{SecTrajGen}
The trajectory generation algorithm starts with the vehicle either at rest or heading towards the goal, $\mathbf{p}_g$, at maximum cruise velocity, $v_c$ (see Sec.~\ref{SecSafetyGuarantee}). As the vehicle moves in the environment, it compiles its sensor inputs to determine the most imminent threats to safety and smoothly adjusts heading and/or velocity accordingly. The vehicle only makes velocity adjustments when there are potentially both vehicle and obstacle collisions.   

\subsection{Ranking Vehicles' Maneuverability}
\label{SubSecVehManeuver}
When two or more vehicles come within communication range of each other, they exchanges cruise velocity, $v_c$, information to determine which vehicles maneuver and which vehicles stay on course. In accordance with Assumption \ref{AssumpLessCapable}, vehicles with larger $v_c$ maneuver around vehicles with smaller $v_c$. If the vehicles have the same $v_c$, then the vehicles are ranked by $ID$. Lower $ID$ values maneuver around vehicles with higher $ID$ values, forming the set $ID_{mnvr} \subseteq ID_{near}$, where $ID_{near} = [ID_1\cdots ID_m]$ is the set of all vehicles within $r_a$. 

\subsection{Obstacle and vehicle collision identification}
\label{SubSecObsVehIdentify}
The vehicle uses distance and angle to determine the most imminent threats to safety. We assume that the sensor provides information equally in all directions. The sensor output is a matrix of angles (relative to vehicle heading) and distances to nearby obstacles. The sensor scan information is used to distinguish different obstacles, each of which is given a unique identifier, $ID$, by the vehicle. The details of that algorithm are not presented here, but the algorithm looks at discontinuities in range and angle to separate the obstacles.

The inter-vehicle communications provide coordinate positions, $\mathbf{p}_j$, velocity, $\mathbf{\dot{p}}_j $, max cruise velocity, $v_c$, and $ID$. The information from the sensor and vehicles is combined in one matrix, $\mathbf{O}$, that tabulates the heading and distance to all the sensed obstacle points and all vehicles in $ID_{mnvr}$. Equation \ref{EqOneMatrixO} defines $\mathbf{O}$, where other vehicles and obstacles are both treated as obstacles:
\begin{equation}
	\label{EqOneMatrixO}
	\mathbf{O} = \left[\begin{array}{ccc}
						\theta_{1} & r_{1} & ID_1\\
						\vdots & \vdots & \vdots\\
						 \theta_{n_1} & r_{n_1} & ID_1 \\

						\vdots & \vdots & \vdots \\
						\theta_{n_1+\cdots +n_{k-1}+1} & r_{n_1+\cdots +n_{k-1}+1} &ID_k \\
						
				\vdots & \vdots & \vdots \\		
						
						 \theta_{n_1+\cdots +n_k} & r_{n_1+\cdots +n_k} &ID_k 
						\end{array}
						\right]
\end{equation}
To determine if there are obstacles along its current heading that violate $r_c$, the vehicle generates two offset vectors parallel to $\mathbf{\dot{p}}$, as shown in Fig.~\ref{FigDetObsInPath}. The relative heading angles to the sensed points from these offset vectors are added to $\mathbf{O}$ to generate $\mathbf{O}_{aug}$:
\begin{equation}
	\label{EqOneMatrixOAug}
	\mathbf{O}_{aug} = \left[ \mathbf{O},~ \left[\begin{array}{cc}
						\theta_{+,1} & \theta_{-,1} \\
						 \vdots & \vdots \\
						\theta_{+,n_1+\cdots +n_k} & \theta_{-,n_1+\cdots +n_k} 
						\end{array}
						\right] \right]
\end{equation}

\noindent The vehicle uses $\mathbf{O}_{aug}$ to identify the $ID$s of the closest sensed point, $ID_r$, and the point most closely aligned to the current or offset heading, $ID_{\theta}$.

\subsection{Heading Change Definition}
\label{SecVehHeadChange}
The obstacles (or vehicles) identified by $ID_r$ and $ID_{\theta}$ are used to determine the heading changes. Each obstacle $ID$ has a corresponding number of sensed points $n_{ID_r}$ and $n_{ID_{\theta}}$ from Eq.~\ref{EqOneMatrixO}. The analysis that follows is for both $ID_r$ and $ID_{\theta}$, but for ease of notation, the $r$ or $\theta$ subscripts are removed.

The first determination is the circumnavigation direction, $\mathbf{z}_{\phi}$, which is held constant while traversing an obstacle and minimizes heading change around obstacles. It is defined as $\pm \mathbf{z}_I$ where $\mathbf{z}_I$ is the inertial frame $z$ axis. The vehicle categorizes obstacles as ``slow'', $v_{o} \leq K_o v_{c}$, or ``fast'', $v_{o} > K_o v_c$ where $0 < K_o < 1$ is a user-defined variable and $v_o = ||\mathbf{\dot{p}}_o||$ is the obstacle velocity magnitude. For stationary and ``slow'' moving obstacles the circumnavigation direction is determined by Eq.~\ref{EqRotSignSlow} and shown in Fig.~\ref{FigDetermineRS}. For avoiding ``fast" moving obstacles, the vehicle goes behind them to reduce unnecessarily lengthy maneuvers as defined in Eq.~\ref{EqRotSignFast}: 
\begin{align}
\label{EqRotSignSlow}
	\mathbf{z}_{\phi,slow} &= \mathrm{sign}\left(\left(\mathbf{p}_{min} \times \mathbf{\dot{p}}_d\right) \cdot \mathbf{z}_I\right) \mathbf{z}_I \\
	\label{EqRotSignFast}
	\mathbf{z}_{\phi,fast} &= \mathrm{sign}\left((\mathbf{\dot{p}}_o \times \mathbf{\dot{p}}_d) \cdot \mathbf{z}_I\right) \mathbf{z}_I 
\end{align}

\begin{figure}
	\begin{center}
		\includegraphics[width=3in]{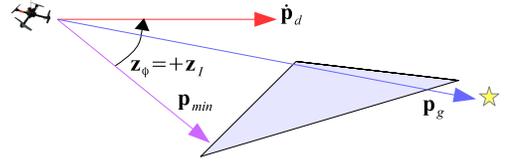}
		\caption{\label{FigDetermineRS} Circumnavigation direction, $\mathbf{z}_{\phi}$, to minimize heading change.}
	\end{center}
\end{figure}

Next, the vehicle examines the current, $\mathbf{p}_{o,i}$, and projected (if obstacle velocity has been estimated), $\mathbf{p}'_{o,i}$, obstacle positions relative to the vehicle. The projected obstacle position is given by Eq.~\ref{EqProjectObs} and the distances to the current and projected obstacle points is given in Eq.~\ref{EqPminCurrent}: 
\begin{align}
	\label{EqProjectObs}
	\mathbf{p}_{o,i}' &= \mathbf{\dot{p}}_{o,i} \Delta T_s + \mathbf{p}_{o,i} \\
	\label{EqPminCurrent}
	r_{i,j} &= ||\mathbf{r}_{i,j}||=||\mathbf{p}_{o,i}-\mathbf{p}_{d}||
\end{align}

\noindent where $~i=1,\cdots, n_{ID}$ and $j=1,2$. The vehicle defines the heading change as (see Fig.~\ref{FigDetermineDeltaPhi})
{\small
\begin{equation}
	\label{EqDelPhiPrime}
	\Delta \phi_l = \left\{\begin{array}{ll}
	\max\limits_{i,j}\left(\cos^{-1} \left(\frac{(\mathbf{p}_{h,i,j} - \mathbf{p}_{d})\cdot \mathbf{\dot{p}}_{d}}{||\mathbf{p}_{h,i,j} - \mathbf{p}_{d}||~||\mathbf{\dot{p}}_{d}||}\right)\right), & \mathbf{z}_{\phi} = \mathbf{z}_I \\
	\\
	\min\limits_{i,j}\left(\cos^{-1} \left(\frac{(\mathbf{p}_{h,i,j} - \mathbf{p}_{d})\cdot \mathbf{\dot{p}}_{d}}{||\mathbf{p}_{h,i,j} - \mathbf{p}_{d}||~||\mathbf{\dot{p}}_{d}||}\right)\right), & \mathbf{z}_{\phi} = -\mathbf{z}_I
	\end{array}
	\right.
\end{equation}
}
\noindent where $l=r,\theta$ and
\begin{align}
	\label{EqPh}
	\mathbf{p}_{h,i,j} &= \mathbf{p}_{d} + p_{h,i,j}\mathbf{R}_{\phi_{h,i,j}} \mathbf{r}_{i,j} \\
	\phi_{h,i,j} &= \sin^{-1} \frac{r_c}{r_{i,j}} \\
		p_{h,i,j} &= \sqrt{(r_{i,j})^2 - r_c^2}		
\end{align}

\noindent where $\mathbf{R}_{\phi_{h,i,j}}$ is the rotation matrix for a $\phi_{h,i,j}$ rotation about $\mathbf{z}_I$. The circumnavigation direction for $\Delta \phi_l$ is
\begin{equation}
	\mathbf{z}_{\Delta \phi_l} = \mathrm{sign}\left(\left(\mathbf{\dot{p}}_{d} \times (\mathbf{p}_{h} - \mathbf{p}_{d})\right) \cdot \mathbf{z}_I\right) \mathbf{z}_I 
\end{equation}
This produces two candidate heading changes, $\Delta \phi_r$ and $\Delta \phi_{\theta}$. The third candidate heading change is to the goal position as given by Eq.~\ref{EqGoalPosDelPhi}, where the circumnavigation direction is given by Eq.~\ref{EqRotSignGoal}:
\begin{align}
	\label{EqGoalPosDelPhi}
			\Delta \phi_g &= \cos^{-1} \left(\frac{(\mathbf{p}_g - \mathbf{p}_{d})\cdot \mathbf{\dot{p}}_{d}}{||\mathbf{p}_g - \mathbf{p}_{d}||~||\mathbf{\dot{p}}_{d}||}\right) \\
\label{EqRotSignGoal}
	\mathbf{z}_{\phi,g} &= \mathrm{sign}\left(\left(\mathbf{\dot{p}}_{d} \times (\mathbf{p}_g - \mathbf{p}_{d})\right) \cdot \mathbf{z}_I\right)\mathbf{z}_I 
\end{align}
The three candidate heading changes are used to determine the actual heading change in Eq.~\ref{EqDelPhiFinal}, and Figure \ref{FigDeltaPhiFinal} shows two example cases. The conditions in Eq.~\ref{EqDelPhiFinal} are evaluated in sequence.
{\small
\begin{multline}
\label{EqDelPhiFinal}
	\Delta \phi = \\
	\left\{\begin{array}{ll}
	\Delta \phi_g, & ||\mathbf{p}_d-\mathbf{p}_g|| < \underset{i}{\mathrm{min}}(r_{i,1}) \\
	\mathrm{max}\left(\Delta \phi_{r},\Delta \phi_{\theta},\Delta \phi_g\right), & \mathbf{z}_{\phi,r} = \mathbf{z}_{\phi,\theta} = \mathbf{z}_I \\
	\mathrm{min}\left(\Delta \phi_{r},\Delta \phi_{\theta},\Delta \phi_g\right) , & \mathbf{z}_{\phi,r} = \mathbf{z}_{\phi,\theta} = -\mathbf{z}_I\\
	\Delta \phi_r, & \Delta \phi_{min} > \Delta \phi_{max} \\
	\Delta \phi_g, & \Delta \phi_{min} \leq \Delta \phi_g \leq \Delta \phi_{max} \\
	\mathrm{argmin} \left(|\Delta \phi_r|,|\Delta \phi_{\theta}|\right), & \mathrm{otherwise}
	\end{array}
	\right.
\end{multline}
}
\noindent where $r_{stop}$ is 
{\small
 \begin{align}
\label{EqDeltaPhiMaxMin1}
	\Delta \phi_{min} &= \Delta \phi_r, ~~\Delta \phi_{max} = \Delta \phi_{\theta},~~\mathrm{for}~\mathbf{z}_{\phi,r} = \mathbf{z}_I \\
	\Delta \phi_{min} &= \Delta \phi_{\theta}, ~~\Delta \phi_{max} = \Delta \phi_r,~~\mathrm{for}~\mathbf{z}_{\phi,r} = -\mathbf{z}_I 		
\end{align}
}

For cases where $\mathbf{z}_{\phi,r} = \mathbf{z}_{\phi,\theta}$ and the maximum heading change corresponds to an obstacle (i.e.~not other vehicles or the goal position), the vehicle also determines if an additional heading change is necessary to match its component velocity in the direction of the obstacle velocity to $v_o$. The vehicle uses $\Delta \phi_l$ from Eq.~\ref{EqDelPhiPrime} to determine the magnitude of the vehicle velocity in the direction of the obstacle velocity, $v_{vo}$: 
\begin{align}
	\mathbf{\dot{p}}_{d}' &= \mathbf{R}_{\Delta \phi_l} \mathbf{\dot{p}}_{d} \\
	v_{vo} &= \mathbf{\dot{p}}_{d}' \cdot \frac{\mathbf{\dot{p}}_{o}}{||\mathbf{\dot{p}}_{o}||}	
\end{align}

\noindent where $\mathbf{R}_{\Delta \phi_l}$ is the rotation matrix for a $\Delta \phi_l$ rotation about $\mathbf{z}_I$. If $v_{vo} < ||\mathbf{\dot{p}}_{o}||$ the vehicle adjusts heading by $\Delta \phi_{vo}$:
\begin{equation}
	\Delta \phi_{vo} = \sin^{-1}\left(\frac{||\mathbf{\dot{p}}_o||}{||\mathbf{\dot{p}}_{d}'||}\right) - \sin^{-1}\left(\frac{v_{vo}}{||\mathbf{\dot{p}}_{d}'||}\right)
\end{equation}

This heading change, $\Delta {\phi_{vo}}$, is then added to $\Delta \phi_l$ to produce new the candidate headings: 
\begin{equation}
	\label{EqPhiDFinal}
	\Delta \phi'_l = \left(\mathbf{z}_I \cdot \mathbf{z}_{\Delta \phi_l}\right) \Delta \phi_l + \left(\mathbf{z}_I \cdot \mathbf{z}_{\Delta \phi_{vo,l}}\right) \Delta {\phi_{vo,l}} 
\end{equation}

\noindent where $l=r,\theta$, and the overall circumnavigation directions, $\mathbf{z}_{\phi,r}$ and $\mathbf{z}_{\phi,\theta}$, are the circumnavigation directions of the larger heading change angle, $\Delta \phi'_l$ or $\Delta \phi_{vo,l}$. The vehicle uses $\Delta \phi'_l$ in Eq.~\ref{EqDelPhiFinal} for the two cases where $\mathbf{z}_{\phi,r} = \mathbf{z}_{\phi,\theta}$ to determine the final $\Delta \phi$.

\begin{figure}
	\begin{center}
		\includegraphics[width=3in]{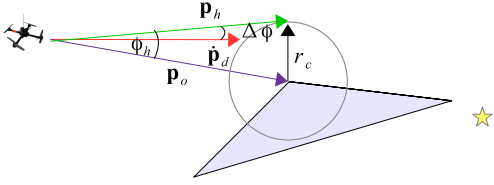}
		\caption{\label{FigDetermineDeltaPhi} Determination of $\Delta \phi$ for all the sensed obstacle points.}
	\end{center}
\end{figure}

\begin{figure}
	\begin{center}
		\includegraphics[width=3in]{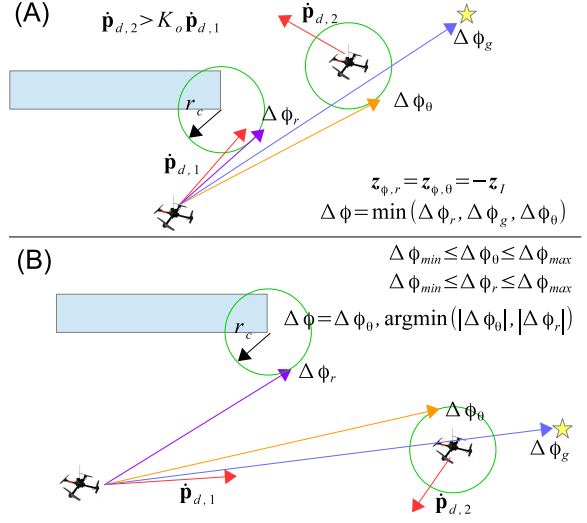}
		\caption{\label{FigDeltaPhiFinal} Example heading change scenarios. (A) The circumnavigation directions are equal so the max or min can be taken. (B) Both $\Delta \phi_r$ and $\Delta \phi_{\theta}$ satisfy the $\Delta \phi_{min}$ and $\Delta \phi_{max}$ constraints. To minimize heading change, $\Delta \phi = \mathrm{argmin} \left(|\Delta \phi_r|,~|\Delta \phi_{\theta}|\right)$.} 
	\end{center}
\end{figure} 

\subsection{Smooth heading and velocity transitions}
\label{SecSigmoid}
The trajectory generation utilizes sigmoid functions to transition from the current heading, $\phi$, and velocity, $v = ||\mathbf{\dot{p}}_{d}||$, to a new heading, $\phi_{n}$, and velocity, $v_{n}$. The hyperbolic tangent function, ($\tanh$), is chosen for its widespread use in generating smooth transitions \cite{Fischer2014}:
\begin{eqnarray}
			\label{EqSigPhi} \phi &=& c_1 \tanh(c_2 \tau - c_3) + c_4 \\
			\label{EqSigV} v &=& d_1 \tanh(d_2 \tau - d_3) + d_4 
\end{eqnarray}

\noindent where $c_i$ and $d_i$ are coefficients to be determined and $\tau$ is the sigmoid curve time (see Sec.~\ref{SecSafetyGuarantee}). The desired velocity vector is then
\begin{equation}
	\label{EqDotPdVec}
	\mathbf{\dot{p}}_d = \left[\begin{array}{c}
					v \cos \phi \\
					v \sin \phi
					\end{array}
					\right]
\end{equation} 

The coefficients can be solved analytically by considering the following assumptions: (1) each sigmoid function occurs over the time interval $\tau=0$ to $\tau=\tau_f$, and (2) since $\tanh$ asymptotically approaches -1 and 1, these are approximated by, -$\varepsilon_1$ and $\varepsilon_1$, (where we use $|\varepsilon_1| = 1-10^{-3}$ to minimize error ($<1$\%) and reduce $\tau_f$). The coefficient solutions are summarized as:
\begin{align}
	\label{EqCoeffSummary1}
	c_3 &= d_3 = \tanh^{-1}-\varepsilon_1 = 3.8\\
	\label{EqCoeffSummaryc2}
	c_2 &= d_2 = 2 c_3/\tau_{f_n} = 7.6/\tau_{f_n} \\
	\label{EqCoeffSummaryc1}
	c_1 &= c_4 = 0.5 \Delta \phi_n \\
	\label{EqCoeffSummaryEnd}
	d_1 &= d_4 = 0.5 \Delta v_n 
\end{align}

The sigmoid curves are summed during navigation so that the vehicle continues to utilize the most recent sensor information. In order to respect the vehicle thrust limitation, successive sigmoid curves match the slope of the previous sigmoid as estimated by a linear approximation as shown in Fig.~\ref{FigSummedSigExplain}. This concept is further defined in Theorem \ref{ThTauF} of Sec.~\ref{SecSafetyGuarantee}.

\begin{figure}
	\begin{center}
		\includegraphics[width=3in]{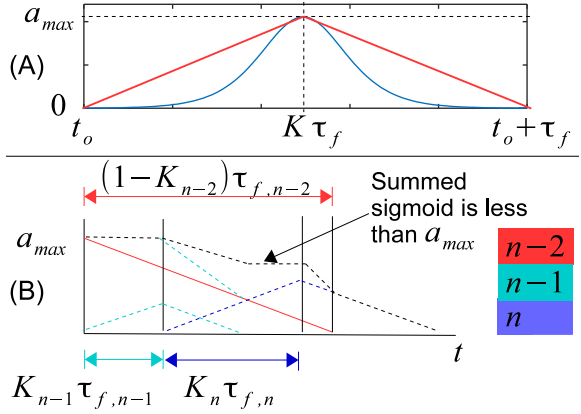}
		\caption{\label{FigSummedSigExplain} (A) The sigmoid curve ``slope" is approximated linearly. (B) Three sigmoid functions are summed together where the ``slopes" of curves $n-1$ and $n$ match the slope of $n-2$. The resulting function does not violate $a_{max}$. }
	\end{center}
\end{figure}

\section{Trajectory Guarantees}
\label{SecSafetyGuarantee}
To guarantee the vehicle can navigate safely in the environment, we present Theorems \ref{ThTauF} and \ref{ThVc}, which define the sigmoid curve timespan and bound the maximum velocity, respectively. 

To aid theorem development we define the available planar force (assumption \ref{AssumptionPlanar}) and the drag force:
\begin{align}
\label{EqFplanarDefinition}
	f_{planar} &= \sqrt{f_{max}^2 - (mg)^2} \\
	\label{EqFdrag}
	\mathbf{f}_{w} &= K_d ||\mathbf{v}_{w}||^2 (-\mathbf{x}_W) \\
	\label{EqKdDefinition}
	K_d &= \frac{1}{2} \rho C_{D} A_{x_W}
\end{align}

\noindent where $m$ is the vehicle mass, $g$ is gravity, $\mathbf{v}_w = \mathbf{\dot{p}} - \mathbf{v}_{air}$ is the resistive wind velocity between the vehicle and the air, $\mathbf{x}_W$ is the wind frame axis aligned with $\mathbf{v}_w$, $\rho$ is the air density, $C_{D}$ is the coefficient of drag, and $A_{x_W}$ is the cross sectional area normal to the resultant drag velocity vector.

\begin{theorem}
\label{ThTauF}
Let $\tau_f$ for the $n^{th}$ sigmoid be defined as
\begin{equation}
\label{EqTauFTh}
\tau_{f,n} = \left\{\begin{array}{ll}
\frac{2c_3}{a_{max}}\sqrt{S_{traj}} , &t_i \geq t_{o,n-1}+\tau_{f,n-1} \\
\sqrt{\frac{2c_3 \sqrt{S_{traj}}}{h_{n-1}}} ,& t_i < t_{o,n-1} + \tau_{f,n-1}
\end{array}
\right.
\end{equation}

\noindent where $t_i$ is the current time, $t_{o,n-1}$ is the previous sigmoid curve offset (Eq.~\ref{EqSigOffset}), $h_{n-1}$ is the approximated linear slope of the previous sigmoid (Eq.~\ref{EqSigSlope}), $v_{w,max} = \mathrm{max}(v_i,v_i + \Delta v) + v_{air}$, $v_i$ is the current velocity, and $S_{traj}$ is a term of the heading and velocity change variables (Eq.~\ref{EqStraj}):
{\small
\begin{align}
	\label{EqSigOffset}
	t_{o,n} &= \mathrm{max}(t_i,t_{o,n-1}+K_{n-1}\tau_{f,n-1}) \\
	\label{EqSigSlope}
	h_n &= \left\{\begin{array}{ll}
	\frac{a_{max}}{(1-K)\tau_{f,n}}, & t_i \geq t_{o,n-1} + \tau_{f,n-1} \\
	h_{n-1}, &  t_i < t_{o,n-1} + \tau_{f,n-1}
	\end{array}\right. \\
	a_{max} &= 1/m\left(f_{planar} - K_d v_{w,max}^2\right) \\
	\label{EqStraj}
	S_{traj} &= (c_1(d_1 H + d_4)(1-H^2))^2 + (d_1(1-H^2))^2
\end{align}
}
\noindent where $H$ is the real solution to
\begin{multline}
  -3c_1^2 d_1^2 H^3 - 5 c_1^2d_1 d_4 H^2 + \\
  (-2 c_1^2 d_4^2 + d_1^2 c_1^2 - 2 d_1^2) H + d_1 d_4 c_1^2 = 0
\end{multline}

\noindent that satisfies $|H| < \varepsilon_1$ and $K = 0.5(\tanh^{-1}(H)/c_3 + 1)$. Then, for this solution for $\tau_f$, the vehicle trajectory does not violate $f_{max}$ in the presence of a bounded disturbance $v_{air}$ that satisfies $v_{air} < \sqrt{f_{planar}/K_d}$.
\end{theorem}

\begin{proof}
The sigmoid curve timespan is constrained by the vehicle's maximum thrust. The planar force is defined as
\begin{equation}
\label{EqFplanarTauFProof}
	\mathbf{f}_{planar} = m\mathbf{\ddot{p}} + \mathbf{f}_{w}
\end{equation}

\noindent where the maximum magnitude of $\mathbf{f}_{planar}$ will be when $\mathbf{\ddot{p}}$ and $\mathbf{f}_w$ are aligned. Each term will be maximized independently which gives a conservative solution for $\tau_f$. Using Eqs.~\ref{EqDotPdVec}, \ref{EqFdrag}, and \ref{EqFplanarTauFProof}, we write the following inequality
\begin{eqnarray}
	||\mathbf{f}_{planar}|| \geq m \sqrt{v^2 \dot{\phi}^2 + \dot{v}^2} + K_d v_{w,max}^2 
\end{eqnarray}

\noindent where $v_{w,max} = \mathrm{max}\left(v_i,v_i + \Delta v\right) + v_{air}$, and $v_i$ is the current velocity. The maximum acceleration from the trajectory will be where $\frac{d||\mathbf{\ddot{p}}_d||}{dt} = 0$. This is expanded as follows:

\begin{align}
	\frac{d||\mathbf{\ddot{p}}_d||}{dt}& = 0 \\
	\frac{d}{dt} \sqrt{v^2 \dot{\phi}^2 + \dot{v}^2} &= 0 \\
	 \label{EqMaxDdotPdSig} v^2 \dot{\phi} \ddot{\phi} +  \dot{\phi}^2 v \dot{v} + \dot{v} \ddot{v} &= 0 
\end{align}

\noindent Equation \ref{EqMaxDdotPdSig} is given in terms of the heading and velocity, but since $\tau_f$ is not known the heading and velocity are also not known. We introduce two new variables, $K$ and $H$, to simplify notation and facilitate a solution. First, $\tau_{max} = K \tau_f$, where $0 \leq K \leq 1$ and $\tau_{max}$ is the value of $\tau$ where $||\mathbf{\ddot{p}}_d||$ is maximum. Typically $K$ will be around 0.5. Second, $H$ defines the common $\tanh$ term in each sigmoid function at the maximum value. The relationships are given as follows:

\begin{align}
	H &= \tanh(c_2\tau - c_3) = \tanh(d_2 \tau - d_3) \\
	&=\tanh \left(\frac{2 c_3}{\tau_f} K \tau_f - c_3 \right) \\
	&= \tanh \left(c_3(2K-1) \right) \\
\label{EqKSol}	K &=\frac{1}{2}\left(\frac{\tanh^{-1}H}{c_3}+1\right)
\end{align}

\noindent Now the sigmoid functions can be substituted into Eq.~\ref{EqMaxDdotPdSig} and simplified as follows:
{\small
\begin{align}
	v^2 \dot{\phi} \ddot{\phi} +  \dot{\phi}^2 v \dot{v} + \dot{v} \ddot{v} &= 0 \\
	(d_1H+d_4)^2\left(c_1c_2(1-H^2)\right)\left(-2c_1c_2^2H(1-H^2)\right) + & \nonumber \\ \left(c_1c_2(1-H^2)\right)^2(d_1H + d_4))\left(d_1 d_2(1-H^2)\right) + & \nonumber \\
	\left(d_1 d_2(1-H^2)\right)\left(-2 d_1 d_2^2 H (1-H^2)\right) &= 0 \\
	\label{EqDerivMax0} -3c_1^2 d_1^2 H^3 - 5 c_1^2d_1 d_4 H^2 + &\nonumber \\ (-2 c_1^2 d_4^2 + d_1^2 c_1^2 - 2 d_1^2) H + &\nonumber \\ d_1 d_4 c_1^2 &= 0
\end{align}
} 
\noindent The final result in Eq.~\ref{EqDerivMax0} is a third order polynomial in $H$. Since all the coefficients are known, the roots can be determined. To be a solution, the roots must be real and satisfy $|H| < \varepsilon_1$.

It should also be noted that the definition for $d_4$ is modified from the sigmoid coefficient definition in Eq.~\ref{EqCoeffSummaryEnd} for this proof to include the current velocity
\begin{eqnarray}
	\label{EqAlternateD4}
	d_4 &=& v_i + \frac{1}{2} \Delta v 
\end{eqnarray}

We now have relationships for the most aggressive part of the trajectory. Next, the drag term is considered where the maximum is:
\begin{equation}
	\label{EqVresDragMax}
	v_{w,max} = \mathrm{max}\left(v_i,~v_i+\Delta v\right) + v_{air}
\end{equation}

To facilitate the solution for $\tau_f$ we also define
\begin{align}
	\label{EqAmax}
	a_{max} &= \frac{1}{m}\left(f_{planar} - K_d v_{w,max}^2\right) \\
	\label{EqStraj}
	S_{traj} &= (c_1(d_1 H + d_4)(1-H^2))^2 + (d_1(1-H^2))^2
\end{align}

Utilizing the solution of $H$ from Eq.~\ref{EqDerivMax0}, the sigmoid function definitions, and Eqs.~\ref{EqVresDragMax}, \ref{EqAmax}, and \ref{EqStraj}, we re-examine Eq.~\ref{EqFplanarTauFProof}, and make the following substitutions
\begin{align}
	&f_{planar} \geq m\sqrt{v^2 \dot{\phi}^2 + \dot{v}^2} + K_d v_{w,max}^2 \\
	&a_{max}^2 \geq c_2^2 S_{traj} \\
	\label{EqTauF1Proof}
	&\tau_{f,n} \geq \frac{2 c_3}{a_{max}}\sqrt{S_{traj}}
\end{align}

Equation \ref{EqTauF1Proof} is utilized to maximize the planar thrust. It is therefore only appropriate when previous sigmoid curves have already completed. When the current sigmoid function is being summed with previous sigmoid functions that have not yet finished, a linear approximation of the previous sigmoid curve slope is used to determine the maximum acceleration and defined as
\begin{equation}
	h_n = \left\{\begin{array}{ll}
	\frac{a_{max}}{(1-K)\tau_{f,n}}, & t_i \geq t_{o,n-1} + \tau_{f,n-1} \\
	h_{n-1}, &  t_i < t_{o,n-1} + \tau_{f,n-1}
	\end{array}\right.
\end{equation}

If the current sigmoid matches the slope of the previous sigmoid, then the maximum thrust will not be exceeded. The maximum acceleration for the current sigmoid is a function of the previous sigmoid slope and $\tau_{f,n}$ and defined by
\begin{equation}
\label{EqAmax2}
	a_{max,n} = h_{n-1} \tau_{f,n}
\end{equation}

Substituting Eq.~\ref{EqAmax2} into Eq.~\ref{EqTauF1Proof} and simplifying produces the following
\begin{align}
	\tau_{f,n} &\geq \frac{2 c_3}{h_{n-1}\tau_{f,n}}\sqrt{S_{traj}} \\
	\tau_{f,n}^2 &\geq \frac{2 c_3}{h_{n-1}}\sqrt{S_{traj}}\\
	\label{EqTauF2Proof}
	\tau_{f,n} &\geq \sqrt{\frac{2 c_3}{h_{n-1}}\sqrt{S_{traj}}}
\end{align}

The two solutions for $\tau_{f,n}$ are summarized in Eq.~\ref{EqTauFTh}.

\end{proof}

\begin{theorem}
\label{ThVc}
Let the vehicle's maximum cruise velocity be defined as
\vspace{-.135in}
\begin{equation}
\label{EqThvc}
v_c = \mathrm{min}\left(v_{c,v},~v_{c,s}\right)
\end{equation}
\noindent where $v_{c,v}$ is the minimum real, positive solution of 
{\small
\begin{equation}
\label{EqFplanarInequalityTh}
	\left(\frac{m}{r_{min}} + K_d\right) v_{c,v}^2 + 2K_d v_{air} v_{c,v} + v_{air}^2 - f_{planar} = 0
\end{equation}
}
\noindent and $v_{c,s}$ is solved simultaneously with the sigmoid curve timespan, $\tau_f$, from the following two equations:
{\small
\begin{align}
\label{EqThVcSenseInequality}
& \int_0^{\tau_f} v_{c,s} \sin \phi(t) dt \leq r_s -r_c - v_{o,max} \tau_f - v_{c,s} \Delta T_s \\
& 
\label{EqVelBoundTF2}
	\tau_{f} = 
	\left\{\begin{array}{ll}
	\frac{ c_3 m\Delta \phi_2 v_c}{f_{planar} - K_d (v_c + v_{air})^2}+\frac{1}{2}\tau_{f,1}, & 2\Delta T_s \geq \tau_{f,1} \\
	\sqrt{\frac{c_3 m\tau_{f,1}\Delta \phi_2 v_c}{2(f_{planar} - K_d (v_c + v_{air})^2)}}+\frac{1}{2}\tau_{f,1}, & 2\Delta T_s < \tau_{f,1}
	\end{array}
	\right.
\end{align}
}

\noindent where the minimum turn radius, $r_{min}$ is user or vehicle defined, $v_{o,max}$ is the expected maximum obstacle velocity, and
{\small
\begin{align}
	& \Delta \phi_1 = \cos^{-1}\left(\frac{r_s - \Delta T_s(v_{o,max} - v_{c,s})}{r_s}\right) + \sin^{-1}\left(\frac{r_c}{r_s}\right) \\
	& \Delta \phi_2 = \pi/2 + \sin^{-1}\left(v_{o,max}/v_{c,s}\right) - \Delta \phi_1 \\
	& \tau_{f,1} = \frac{ c_3 m\Delta \phi_1 v_{c,s}}{f_{planar} - K_d (v_{c,s} + v_{air})^2}
\end{align}
}
\noindent Then, for this solution for $v_c$ the vehicle does not violate $f_{max}$ when making a turn of radius $ r_t \geq r_{min}$ for a bounded disturbance $v_{air}$ applied in any direction that satisfies $v_{air} < \sqrt{f_{planar}/K_d}$.
\end{theorem}

\begin{proof}
The first constraint on $v_c$ is due to the vehicle thrust limitations. The maximum values for each of the components in Eq.~\ref{EqFdesVec} are considered. 
\begin{equation}
	\label{EqFdesVec}
	\mathbf{f}_{d} = m \mathbf{g} + \mathbf{R}_{WI}\mathbf{f}_{w} + \overbrace{\mathbf{f}_{n} + \mathbf{f}_{t}}^{m\mathbf{\ddot{p}}}
\end{equation}

The worst case drag force, $f_{w,max}$ occurs when the vehicle is traveling at $v_c$ into the wind. Likewise, the maximum normal force occurs for the vehicle's tightest turning radius $r_{min}$. Finally, because the vehicle is at its maximum cruise velocity, $\mathbf{f}_{t,max} = 0$. The maximum values of each of the components are summarized in Eqs.~\ref{EqFwindmaxMag} to \ref{EqFtangMaxMag}. 
\begin{align}
	\label{EqFwindmaxMag}
	f_{w,max} & =  \frac{1}{2} \rho C_{D} A_{x_W} v_{w-max}^2 = K_d v_{w-max}^2 \\
	f_{n,max} & =  m \frac{v_c^2}{r_{min}} \\
	\label{EqFtangMaxMag}
	f_{t,max} &= 0
\end{align}

\noindent where $v_{w,max} = v_c + {v}_{air}$ and we assume that the area normal to $x_W$, $A_{x_W}$, and the drag coefficient, $C_D$ are known. 

The planar force vector is defined by
\begin{equation}
	\mathbf{f}_{planar} = \mathbf{f}_n + \mathbf{f}_{w}
\end{equation}

\noindent where the maximum magnitude of $\mathbf{f}_{planar}$ occurs when $\mathbf{f}_n$ and $\mathbf{f}_w$ are aligned: 

\begin{equation}
\label{EqFplanarInequalityProof}
	||\mathbf{f}_{planar}|| \geq \frac{mv_c^2}{r_{min}} + K_d \left(v_c + v_{air}\right)^2
\end{equation}

\noindent Equation \ref{EqFplanarInequalityProof} is equivalent to Eq.~\ref{EqFplanarInequalityTh}.

Since Eq.~\ref{EqFplanarInequalityTh} is a quadratic in $v_c$, we can write 
\begin{equation}
	a_{vc} v_c^2 + b_{vc} v_c + c_{vc} =0
\end{equation}

\noindent
\begin{align}
	a_{vc} &= \frac{m}{r_{min}} + K_d \\
	b_{vc} &= 2 K_d v_{air} \\
	c_{vc} &= K_d v_{air}^2 - f_{planar}
\end{align}

\noindent The roots are then  
\begin{equation}
	v_c = \frac{ -b_{vc} \pm \sqrt{b_{vc}^2 - 4 a_{vc} c_{vc}}}{2 a_{vc}} 
\end{equation}

\noindent The solution for $v_c$ must be real and positive which means $b_{vc}^2 - 4 a_{vc} c_{vc} \geq 0$ and $-b_{vc} + \sqrt{b_{vc}^2 - 4 a_{vc} c_{vc}} > 0$. Re-arranging these two inequalities gives
\begin{align}
	4a_{vc} c_{vc} &\leq b_{vc}^2 \\
	\label{EqLimitingCase}
	4a_{vc} c_{vc} &< 0
\end{align}

\noindent which shows that the second inequality is the more restrictive constraint. Since $a_{vc} > 0$, Eq.~\ref{EqLimitingCase} reduces to $c_{vc} \leq 0$. Solving for $v_{air}$ gives

\begin{align}
	c_{vc} &\leq 0 \\
	K_d v_{air}^2 - f_{planar} &\leq 0 \\
	\label{EqVairMax}
	v_{air}& \leq \sqrt{\frac{f_{planar}}{K_d}}
\end{align}

The second constraint on $v_c$ is due to sensor limitations. We consider a vehicle traveling towards an obstacle where the velocity vector of the vehicle is opposite the velocity vector of the obstacle. In the worst case scenario the vehicle is $r_s + \varepsilon$ away from the obstacle and thus does not sense it. After $\Delta T_s$ the sensor will identify the obstacle and a heading change determined. The heading change is
\begin{equation}
	 \Delta \phi_1 = \cos^{-1} \left(\frac{r_s - \Delta T_s(v_{o,max} + v_c)}{r_s}\right) + \sin^{-1} \frac{r_c}{r_s}
\end{equation}

After $2\Delta T_s$ the vehicle makes an estimate of the obstacle velocity and determines the remaining heading change to match the component of the vehicle velocity in the obstacle's direction to the obstacle velocity. This heading change is
\begin{equation}
	\Delta \phi_2 = \pi/2 + \sin^{-1}\left(\frac{v_o}{v_c}\right) - \Delta \phi_1
\end{equation}

Since the minimum distance between the vehicle and obstacle monotonically approaches $r_c$, the following inequality must hold 
\begin{equation}
\label{EqSensorIntInequality}
	\int_0^{\tau_f} v_{c} \sin \phi(t) dt \leq r_s - v_{o,max} \tau_f - v_{c} \Delta T_s - r_c
\end{equation}

\noindent when Eq.~\ref{EqSensorIntInequality} is re-arranged it is equivalent to Eq.~\ref{EqThVcSenseInequality}. This equation has two unknowns in $v_{c}$ and $\tau_f$. The solution for $\tau_f$ is dependent on Eq.~\ref{EqTauFTh}. For the initial heading change, $t \geq t_{o,n-1} + K_{n-1}\tau_{f,n-1}$ since there is no $n-1$ sigmoid curve. Additionally since $\Delta v =0$, $S_{traj}$ simplifies to $(c_1 d_4)^2 = (c_1 v_c)^2$. Since $v_c$ is unknown, $a_{max}= 1/m(f_{planar} - K_d(v_c + v_{air})^2)$. Eq.~\ref{EqTauFTh} can be simplified as:
\begin{equation}
\label{EqVelBoundTauF}
\tau_f = \frac{2 c_3 m c_1 v_c}{f_{planar} - K_d (v_c + v_{air})^2}
\end{equation}

\noindent Substituting Eq.~\ref{EqCoeffSummaryc1} into Eq.~\ref{EqVelBoundTauF} gives the initial timespan as
\begin{equation}
\label{EqVelBoundTF1}
		\tau_{f,1} = \frac{ c_3 m\Delta \phi_1 v_c}{f_{planar} - K_d (v_c + v_{air})^2}
	\end{equation}
	
\noindent The second sigmoid heading change of $\Delta \phi_2$ will also be solved by Eq.~\ref{EqTauFTh}, but it is unknown whether $\tau_{f,1} \geq 2 \Delta T_s$. If it is not, then:
\begin{equation}
	h_{n-1} = \frac{1/m(f_{planar} - K_d (v_c + v_{air})^2)}{1/2 \tau_{f,1}}
\end{equation}

\noindent and the sigmoid curve timespan is defined as
\begin{multline}
\label{EqVelBoundTF2}
	\tau_{f} = \\
	\left\{\begin{array}{ll}
	\frac{ c_3 m\Delta \phi_2 v_c}{f_{planar} - K_d (v_c + v_{air})^2}+\frac{1}{2}\tau_{f,1}, & 2\Delta T_s \geq \tau_{f,1} \\
	\sqrt{\frac{c_3 m\tau_{f,1}\Delta \phi_2 v_c}{2(f_{planar} - K_d (v_c + v_{air})^2)}}+\frac{1}{2}\tau_{f,1}, & 2\Delta T_s < \tau_{f,1}
	\end{array}
	\right.
\end{multline}

\noindent Equations \ref{EqSensorIntInequality} and \ref{EqVelBoundTF2} must be solved simultaneously for $v_{c}$ and $\tau_f$.

Once both constraints have been considered, the cruise velocity is the minimum value as defined by Eq.~\ref{EqThvc}.
\end{proof}

\subsection{Goal Position Convergence}
The vehicle continues to head towards the goal position, $\mathbf{p}_g$, and once $||\mathbf{\dot{p}}_g|| = 0$, the vehicle reaches the goal position in finite time. We define $e_{\phi_g} = \phi_g - \phi$ as the error in the heading angle towards the goal position and $r_g$ as the distance to the goal position. The following statements can be made:

\begin{enumerate}
	\item When the vehicle starts moving it is headed towards the goal position, $e_{\phi_g} = 0$, and is some distance, $r_g > 0$ away. 
	\item The vehicle maneuvers around obstacles and other vehicles and from Assumption \ref{AssumpLessCapable}, $|e_{\phi_g}| > 0$ for a finite time. Once the obstacles have been cleared $e_{\phi_g} \rightarrow 0$ and $r_g \rightarrow 0$.
\end{enumerate}

\section{Vehicle and Controller}
\label{SecVehController}
The vehicle dynamics for a quadrotor are given in Eqs.~\ref{EqEOMLin} and \ref{EqEOMAng}. Equation \ref{EqEOMLin} is written in the inertial frame, and Eq.~\ref{EqEOMAng} is written in the body frame: 
\begin{align}
	\label{EqEOMLin}
	m\mathbf{\ddot{p}} &= \mathbf{f} + m\mathbf{g} + \mathbf{d}_{p} \\
	\label{EqEOMAng}
	\mathbf{J}\mathbf{\dot{\omega}} &= \mathbf{\omega} \times \mathbf{J} \mathbf{\omega} + \mathbf{u} + \mathbf{R}_{IB}\mathbf{d}_{\omega}
\end{align}

\noindent where $\mathbf{f}$ is the total thrust, $\mathbf{d}_p$ is the translational disturbance (including drag), $\mathbf{J}$ is the vehicle moment of inertia, $\mathbf{\dot{\omega}}$ is the rotational acceleration, $\mathbf{u}$ is the total torque, $\mathbf{R}_{IB}$ is the rotation matrix from the inertial to body frame, and $\mathbf{d}_{\omega}$ is the rotational disturbance. The control inputs are the vehicle force, $\mathbf{f}$, and torque, $\mathbf{u}$. 

The vehicle dynamics also include aerodynamic effects on the propellers like thrust reduction from propeller inflow velocity \cite{Leishman2006} and blade flapping \cite{Hoffmann}.

The vehicle controller uses an inner- and outer-loop control similar to \cite{Bialy2013, Cao2016} where the outer loop controls translation and the inner loop controls rotation. The outer loop uses a nonlinear robust integral of the sign of the error (RISE) controller \cite{FischerNL2014} (Eqs.~\ref{EqFischerControl} to \ref{EqFischere2}) and the inner loop uses PID control \cite{Cao2016} (Eq.~\ref{EqPID}):
\begin{align}
	\label{EqFischerControl}
	\mathbf{f} &= (k_s + 1) \mathbf{e}_2 - (k_s + 1) \mathbf{e}_2(0) + \nu \\
	\mathbf{\dot{\nu}} &= (k_s + 1) \alpha_2 \mathbf{e}_2 + \beta \mathrm{sign}(\mathbf{e}_2) \\
		\label{EqFischere2}
	\mathbf{e}_2 &= \mathbf{\dot{e}}_1 + \alpha_1 (\mathbf{p}_d - \mathbf{p}) \\
	\label{EqPID}
	\mathbf{u} &= k_p \mathbf{q}_d + k_i \int \mathbf{q}_d dt + k_d \mathbf{\dot{q}_d}
\end{align}

\noindent where $k_s >0$ and $\alpha_2 > 1/2$ are translational control gains, and $k_p ,k_i,k_d > 0$ are PID control gains for desired Euler angles, $\mathbf{q}_d$, determined from $\mathbf{f}$.

\section{Simulation Results}
\label{SecSim}
To demonstrate the algorithm capabilities, we show a scenario where two vehicles navigate into a building to different goal positions. There is a bounded mean wind disturbance of 2 m/s outside the building, a transition zone on entering the building, and no wind inside. The wind field uses the Von K\'arm\'an power spectral density and is spatially correlated \cite{Cole2013}. Both vehicles have the same parameters: $f_{max} = 10.2$N, $m$ = 0.54kg, $J = \mathrm{diag}([0.0017,0.0017,0.0031])$ kg/m$^2$, $C_d$ = 1.7, $r_s$ = 10m, $\Delta T_s = 1$s, $r_c$ = 2m, and $A_{x_w}= 0.2$m$^{2}$. The maximum cruise velocity for both vehicles is solved from Theorem 2 as $v_c$ = 1.83 m/s, and from Sec.~\ref{SubSecVehManeuver} vehicle 1 maneuvers around vehicle 2. Figure \ref{FigSimEx1OverView} shows an overview of the vehicles' trajectories, Fig.~\ref{FigExSim} shows snapshots of vehicle navigation, and Fig.~\ref{FigSimEx1VehPhi} shows smooth heading changes. 

The vehicle clears the obstacle by greater than $r_c$, and the thrust constraint is not violated. The computation time to take the sensor input and generate a trajectory is approximately 0.5 seconds for $>160$ sensor points when run on a laptop computer (Matlab 2015b, 2.8GHz processor, 8 GB RAM). It is expected that the computation time would be significantly reduced if implemented as compiled code.
\begin{figure}
	\begin{center}
		\includegraphics[width=3in]{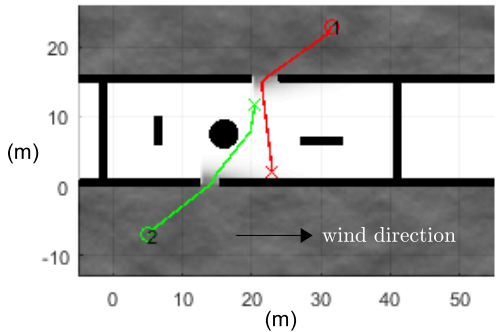}
		\caption{\label{FigSimEx1OverView} Overview of vehicles moving into a building in the presence of a bounded wind disturbance (shown here at one instance in time). The vehicles clear all obstacles by $r_c$ and do not violate $f_{max}$.}
	\end{center}
\end{figure}

\begin{figure}
	\begin{center}
		\includegraphics[width=3in]{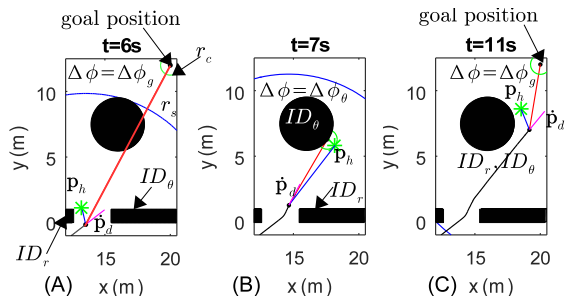}
		\caption{\label{FigExSim} Snapshots of the vehicle maneuvering in the environment. (A) The vehicle navigates through a window or door in the building. (B) The vehicle identifies the next obstacle to maneuver around. (C) The vehicle has a clear line to the goal position, safely clearing the obstacle.}
	\end{center}
\end{figure}

\begin{figure}
	\begin{center}
		\includegraphics[width=2.5in]{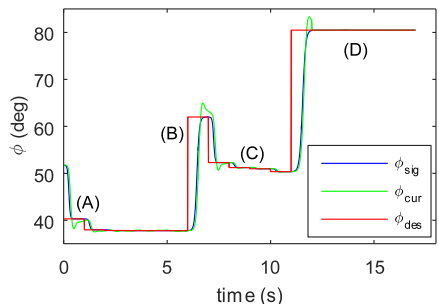}
		\caption{\label{FigSimEx1VehPhi} Smooth heading changes for the vehicle as it navigates an unknown environment at $v_c$. (A) Initial adjustment on detection of the obstacle, (B) Heading change to come through the building opening (C) Heading adjustment for the circular obstacle, and (D) Heading change toward goal position.}
	\end{center}
\end{figure}

\section{Conclusion and Future Work}
\label{SecConclusion}
The trajectory generator presented navigates a vehicle in an unknown environment while avoiding obstacles and other vehicles and respecting the vehicle's physical limitations. The vehicle uses its sensor and communication inputs to compute heading changes to avoid obstacles by a prescribed distance. The sigmoid functions used to transition heading and velocity provide smooth motion and incorporate the heading changes from each sensor update by matching the sigmoid slopes and summing the curves. Similarly, the vehicle incorporates the estimated wind disturbance, thrust limitations, and sensor constraints to solve for the sigmoid curve time intervals and bound the maximum safe cruise velocity. The simulation demonstrates these properties, showing smooth transitions and respecting maximum required force.

The trajectory generation presented could be extended to 3D motions by rotating the plane in which the vehicle traverses, or combining the planar motion described with a separate altitude trajectory. The thrust required for altitude adjustment could be accounted for independently, thus reducing the thrust available for planar motion. The combination of the planar and altitude trajectories would produce a 3D trajectory that respects the thrust constraints. Additional areas for exploration include relaxing the assumption of perfect sensor information, including rotational disturbances, and incorporating this trajectory generator into a higher level formation controller. 

\bibliographystyle{unsrt}
\bibliography{C:/Users/cole/Documents/ColeThesisWork/Documentation/ReferencesBibFiles/Controls}

\end{document}